\documentclass[preprint,12pt]{elsarticle1}
\RequirePackage[OT1]{fontenc}
\RequirePackage[numbers]{natbib}
\usepackage{dcart}
\usepackage{graphicx}

\usepackage[colorlinks=true,
        raiselinks=true,
        linkcolor=MidnightBlue,
        citecolor=Mahogany,
        urlcolor=ForestGreen,
        plainpages=false]{hyperref}

\DeclareMathOperator{\id}{id}

\allowdisplaybreaks

\def \non{{\nonumber}}

\newcommand{\rt}{\rightarrow}

\newcommand{\PSet}{\mathcal{P}}

\newcommand{\Lp}[1]{\mathrm{L}_{#1}}

\newcommand{\metsp}{U}

\newcommand{\g}{\gamma}

\newcommand{\oo}{\infty}
\newcommand{\s}{\sigma}
\def\SC{\mathcal}

\def \triple|{|\! | \! |}
\def\Ball{B}
\def\ort{\mathcal{K}}
\def\f{\frac}

\def\<{\langle}
\def\>{\rangle}
\def\~{\tilde}

\def\N{\mathbb N}
\def\Z{\mathbb Z}
\def\R{\mathbb R}

\journal{}

\begin{document}

\begin{frontmatter}
\title{Uniform moment bounds of multi-dimensional functions of discrete-time stochastic processes}

\author{Arnab Ganguly\corref{t1}\fnref{a1}}
\ead{gangulya@control.ee.ethz.ch}
\author{Debasish Chatterjee\corref{t2}\fnref{a2}}
\ead{chatterjee@sc.iitb.ac.in}
\author{John Lygeros\corref{t2}\fnref{a1}}
\ead{lygeros@control.ee.ethz.ch}
\author{Heinz Koeppl\corref{t1}\fnref{a1}}
\ead{koeppl@ethz.ch}

\cortext[t1]{A.\ Ganguly and H.\ Koeppl were partially supported by the Swiss National Science Foundation, grant PP00P2\_128503/1}
\cortext[t2]{D.\ Chatterjee and J.\ Lygeros were partially supported by the European Commission under the project MoVeS, FP7-ICT-2009-257005, and the HYCON2 Network of Excellence (FP7-ICT-2009-5).}


\address[a1]{Automatic Control Laboratory, ETH Z\"urich\\ Physikstrasse 3, 8092 Z\"urich,\\ Switzerland}
\address[a2]{Systems \& Control Engineering\\IIT-Bombay, Powai\\ Mumbai 400~076\\ India}

	\begin{abstract}
		We establish conditions for uniform $r$-th moment bound of certain $\R^d$-valued functions of  a discrete-time stochastic process taking values in a general metric space. The conditions include an appropriate negative drift  together with a uniform $\Lp p$ bound on the jumps of the process for $p > r + 1$. Applications of the result are given in connection to iterated function systems and biochemical reaction networks.
	\end{abstract}


\begin{keyword}
stochastic stability\sep Markov chains\sep uniform moment bounds\sep invariant distributions\sep stochastic control 
\vspace{.3cm}
\MSC Primary: 60G07 \sep 60J10; secondary: 60J20, 93E15
\end{keyword}

\end{frontmatter}

	\section{Introduction}
                  Stability is an important property in any form of dynamical systems. For deterministic dynamical systems, stability is mainly concerned with different types of behavior of the trajectories of the system which start near the equilibrium point. For the stochastic counterpart, many notions of stability have been developed in the context of Markov chains or more generally Markov processes. Typically, the study of stability of a Markov chain involves checking the existence of invariant measures and investigating  various types of convergence of the transition kernels to the invariant measure. Further investigation involves seeking criteria for ergodicity, Harris recurrence or positive Harris recurrence. While different types of Lyapunov techniques are used for studying stability in the deterministic case,  the corresponding investigation for Markov chains  is carried out by suitable uses of Foster-Lyapunov functions. The essence of the matter is the following: given a process $\{X_n\}_{n\in\Nz}$ taking values in a Polish space $U$, one constructs a non-negative measurable function $V:U\rt [0,\infty)$, called a Foster-Lyapunov function, such that the process $\{V(X_n)\}_{n\in\Nz}$ possesses certain desirable properties, e.g, some kind of Foster-Lyapunov drift condition. The process $\{V(X_n)\}_{n\in\Nz}$, being real-valued and nonnegative, often admits easier analysis and standard results yield various conclusions about recurrence, ergodicity or rate of convergence of measures, etc, for the original process $\{X_n\}_{n\in\Nz}$.  A good reference for various Foster-Lyapunov drift conditions for discrete time  Markov chains  is \citep{ref:MeyTwe-09}. For various results concerning invariant measures of Markov chains, see \citep{ref:Her-LerLas-03} and \citep{ref:Zahar05} for general Markov-Feller operators. For continuous time Markov processes,  \citep{ref:Kha-80} and \citep{ref:Mao97} discuss various techniques for checking stochastic stability.
                  
                  In this paper we consider a different notion of stability, namely, uniform moment bounds for  multi dimensional  functions of discrete time stochastic processes. More precisely, given a stochastic process $\{X_n\}_{n\in\Nz}$ taking values in  a metric space $\metsp$ and a sequence of functions $\{G_n:\metsp \rt \R^d\}$, conditions are sought such that $\sup_n\EE[\|G_n(X_n)\|^r] <\infty$. Uniform moment bounds of stochastic processes or functions of stochastic processes have important applications in several disciplines like queueing theory, control theory, physics, etc. For an $\R$-valued process $\{X_n\}$, Pemantle  and Rosenthal \citep{ref:PemRos-99} established conditions for $\sup_n\EE[(X_n^+)^r]$ to be finite. The conditions involve a ``constant'' negative drift  together with a uniform $\Lp p$ bound on the jumps of the process for $p > r + 1$. The fact that the result does not require existence of Lyapunov functions makes it particularly useful, as explicit  construction of suitable Lyapunov functions is often a difficult task \citep{ref:BerFriGol-03}. For a Markov chain $\{X_n\}$ taking values in a general metric space, \citep{ref:ChaPal-11} used the theory of excursions of Markov processes to establish a uniform $\Lp 1$ bound on an $\R$-valued function of $X_n$. Their hypotheses require the existence of  a certain derived supermartingale with a prescribed rate of decay  when the process stays outside a compact set. While this approach does not work directly with drift conditions as in the Foster-Lyapunov function approach, the existence of the desired supermartingale is in general not straightforward to verify.  
                  
                  Our paper generalizes the one-dimensional result of \citep{ref:PemRos-99} in two directions: first, we consider $\R^d$-valued functions of the stochastic process $\{X_n\}$ taking values in a general metric space;  second, the drift condition is generalized to incorporate a number of scenarios. More precisely, our main theorem reads as follows:
                  \begin{theorem}
		\label{mainth}
		\label{t:main}
			Let $(\Omega, \sigalg, \{\sigalg_n\}, \PP)$ be a filtered probability space, $\metsp$ a complete and separable metric space and $C\subset \metsp$ (Borel) measurable.  Let $\{G_n:\metsp \rt \R^d\}$ and $\{H_n:\metsp \rt \R^d\}$ be  sequences of measurable functions satisfying
			\begin{enumerate}[label={\rm (\roman*)}, leftmargin=*, align=right]
			\item \label{condg1}for every $n$, $G_n, H_n : \metsp\setmin C \rt \R^d_+$;
			\item \label{condg2}for every $n$, $G_n^{-1}G_n( \metsp\setmin C) = \metsp\setmin C$;
			\item \label{condg3} there exist constants $a,b>0$, such that $a\leq\inf_n\inf_{x\notin C}\|H_n(x)\|, \\ \sup_n\sup_{x\notin C}\|H_n(x)\|\leq b$ and $\sup_n\sup_{x\in C}\|G_n(x)\| \leq b$.
			\end{enumerate}
			 Let $\{X_n\}$ be a sequence of $\{\SC{F}_n\}$-adapted  $\metsp$-valued random variables. Assume that $X_0\in C$ and the following two conditions hold:
			\begin{enumerate}[label={\rm (\roman*)}, leftmargin=*, align=right,start=4]
						\item for all $n\geq0$,
					\begin{equation}
					\label{cond1}
						\EE\bigl[G_{n+1}(X_{n+1})-G_n(X_n)\,\big|\,\SC{F}_n\bigr] \leq - H_n(X_n) \quad \text{on}\quad \{X_n \notin C\};
					\end{equation}
				\item \label{2condition} there exist constants $L > 0$ and $p>2$ such that for all $n\geq0$
					\begin{align}
					\label{cond2}
						\EE\bigl[\|G_{n+1}(X_{n+1})-G_n(X_n)\|^p\,\big|\,\SC{F}_n\bigr] \leq L.
					\end{align}
			\end{enumerate}
			Then for any $0<r<p-1$, there exists a constant $\eta \Let \eta(p,a,b,d,L,r)$ such that 
			\[
				\sup_{n\in\Nz} \EE\bigl[\|G_n(X_n)\|^r \bigr] \leq \eta.
			\]
			If in \ref{2condition}, instead of \eqref{cond2} we have 
			\begin{align}
					\label{cond2_ii}
						\EE\bigl[\|G_{n+1}(X_{n+1})-G_n(X_n)\|^p\bigr] \leq L, 
					\end{align}
		        then  for any $0<r<p/2-1$, there exists a constant $\eta \Let \eta(p,a,b,d,L,r)$ such that 
			\[
				\sup_{n\in\Nz} \EE\bigl[\|G_n(X_n)\|^r\bigr] \leq \eta.
			\]

		\end{theorem}

		Here, $\N \Let \{1, 2, \ldots \}$,  $\Nz \Let \{0\}\cup\N$, $\R^d_+\equiv \{x\in \R^d:x_i \geq 0,i=1,\hdots,d\}$ and $\left\|\cdot\right\|$ denotes the Euclidean norm on $\R^d$ . Typically in many applications, $G_n\equiv G$ is a continuous function and $C\subset \metsp$ is compact. Therefore the condition $\sup_n\sup_{x\in C}\|G_n(x)\| < b$ automatically holds. In fact, Theorem \ref{t:main} also holds if the condition  $\sup_n\sup_{x\in C}\|G_n(x)\| < b$ is replaced by the condition $\sup_n \EE[\|G_n(X_n)\|^p1_{\{X_n \in C\}}] < \infty$. 
		
		Next, note that, since $G_n^{-1}G_n( \metsp\setmin C) \supset \metsp\setmin C$, \ref{condg2} is equivalent to requiring $G_n^{-1}G_n( \metsp\setmin C) \subset \metsp\setmin C$. A necessary and sufficient condition for \ref{condg2} is $G_n( \metsp\setmin C)\cap G_n(C) =\emptyset$.
		
		The inequalities between the various vectors in Theorem \ref{t:main} are interpreted component-wise, i.e., for $x,y \in \R^d$, we have $x\leq y$ if $x_i \leq y_i$ for all $i=1,\ldots,d$. 
 One salient point to note in Theorem \ref{t:main} is that no Markovian assumption on the process $\{X_n\}$ is made. The component-wise inequality used in \eqref{cond1} is the natural partial order in the first orthant $\R^d_+$. However, $\R^d_+$ plays no special role in the proof of Theorem \ref{t:main} and the  result for a general orthant is stated in Theorem \ref{t:main_orth} with the partial order of $\R^d_+$ replaced by an appropriate partial order of the orthant considered. 
 
            Our result can be particularly helpful in queueing theory, control theory where a uniform bound on the variance of the states of a multi-dimensional stochastic system is desirable. Section \secref{s:app3} outlines a method for obtaining uniform moment bounds of multi-dimensional iterated function systems, an important area in the field of control theory. Section \secref{s:app4} concerns applications in connection to general biochemical reaction networks.
            
            Finally, we wish to mention that for Markov processes uniform moment bounds often imply existence of an invariant probability measure. More generally,  as the discussion after the proof of Theorem \ref{t:main_orth} shows that a uniform moment bound of an appropriate function of the Markov process leads to the existence of an invariant probability measure. Thus, the central theme of our paper is very much related to the traditional notion of stochastic stability.

	\section{Proof of Theorem \ref{t:main}}
	\label{s:mainres}
All the analysis hereafter assumes the existence of a probability space defined in the hypotheses of Theorem \ref{t:main}.

We start with analogue of \citep[Lemma 7]{ref:PemRos-99}. The proof is just a simple application of the following version of Burkholder's inequality \cite[\S6.3]{ref:Str-11} and follows exactly the same steps as in \citep[Lemma 7]{ref:PemRos-99}. 
For an $\{\SC{F}_t\}$-martingale $\{M_t\}$ taking values in $\R^d$, let $\{[M]_t\}$ denote its scalar quadratic variation process (see \citep[Chap 2]{ref:EthKur-86}).
\begin{lemma}
\label{burk}
Let $\{M_t\}$ be an $\{\SC{F}_t\}$-martingale taking values in $\R^d$. Then for $1\leq p<\infty$, there exists a constant $c_p>0$ such that
$$\EE[\|M_t-M_s\|^p|\SC{F}_s] \leq c_p\EE[([M]_t -[M]_s)^{p/2}|\SC{F}_s], \ \ 0\leq s<t.$$
\end{lemma}

\begin{lemma}
\label{burkapp}
Let $\{M_n\} \subset \R^d$ be an $\{\SC{F}_n\}$-martingale. Assume that for some  $p>2$, there exists a sequence of constants $\nu_n$ such that
$$\EE[\|M_{n+1}-M_{n}\|^p|\SC{F}_n] \leq \nu_n\quad  \text{for all } n \geq 0.$$
Then there exists a constant $c_p$ such that $\EE[\|M_n -M_k\|^p|\SC{F}_k] \leq  c_p (n-k)^{p/2-1}\sum_{j>k}^n \nu_j.$
\end{lemma}

\begin{proof}
Notice that by Burkholder's inequality there exists a constant $c_p$ such that
\[\EE[\|M_n -M_k\|^p|\SC{F}_k]  \leq c_p\EE[(\sum_{j>k}^n\|M_j - M_{j-1}\|^2)^{p/2}|\SC{F}_k]\]
By Holder's inequality,
\[\|\sum_{i=1}^na_i\|^{p/2}\leq n^{p/2-1}\sum_{i=1}^na_i^{p/2}.\]
Taking $a_j = \|M_j-M_{j-1}\|^2$, it follows that
\begin{align*}
\EE[\|M_n -M_k\|^p|\SC{F}_k]&  \leq c_p (n-k)^{p/2-1}\sum_{j>k}^n\EE[\|M_j - M_{j-1}\|^p|\SC{F}_k]\\
& \leq c_p (n-k)^{p/2-1}\sum_{j>k}^n \nu_j.
\end{align*}

\end{proof}

\begin{remark}\label{burkapp2}
If in Lemma \ref{burkapp} we have the  weaker hypothesis:
$$\EE[\|M_{n+1}-M_{n}\|^p] \leq \nu_n\quad  \text{for all } n \geq 0,$$
then $\EE[\|M_n -M_k\|^p] \leq  c_p (n-k)^{p/2-1}\sum_{j>k}^n \nu_j.$
\end{remark}

\begin{lemma}
\label{martbd}
Let $\{M_n\} \subset \R^d$ be an $\{\SC{F}_n\}$-martingale with $\EE[\|M_0\|^p] < \infty$ and assume that for some  $p\geq 2$, there exists a constant $\nu$ such that
$$\EE[\|M_{n+1}-M_{n}\|^p] \leq \nu\quad \text{for all } n \geq 0.$$
Then for $0<r<p$, there exists a constant $\theta\Let \theta(\EE[\|M_0\|^p],\nu,p,r)$ such that
$$ \EE\bigl[\|M_n\|^r\indic{\{\|M_{n}\| \geq n\}}\bigr] \leq \theta/ n^{p/2-r}.$$
\end{lemma}

\begin{proof}
First notice that by  Remark \ref{burkapp2} with $\nu_n\equiv \nu$, we have $\EE[\|M_n -M_0\|^p] \leq c_p\nu n^{p/2}$. Hence,
\begin{align}\label{EMnbound}
\EE[\|M_n\|^p] \leq 2^p(\EE[\|M_n - M_0\|^p] + \EE[\|M_0\|^p]) \leq \theta_0n^{p/2}
\end{align}
where $\theta_0$ is a constant depending on $\EE[\|M_0\|^p],\nu$ and $p$.
Next, notice that 
\begin{align*}
\EE[\norm{M_n}^r\indic{\{\norm{M_n} \geq n\}}]& = n^r \PP(\norm{M_n} \geq n) + \int_{n}^\oo ry^{r-1}\PP(\norm{M_n} > y) \,\drv y\\
& \leq n^{r-p}\EE[\norm{M_n}^p] +\int_{n}^\oo ry^{r-1-p}\EE[\norm{M_n}^p] \,\drv y \\
& \leq \EE[\norm{M_n}^p)\left(n^{r-p} +\int_n^\oo ry^{r-1-p} \,\drv y\right)\\
& = \EE[\norm{M_n}^p]\left( n^{r-p} +\frac{r}{p-r}n^{r-p}\right)\qquad\text{since } r<p,\\
& \leq \theta n^{r-p}n^{p/2}\quad \text{for some $\theta > 0$},
\end{align*} 
where for the last inequality the bound for $\EE[\norm{M_n}^p]$ from  \eqref{EMnbound} is used.
\end{proof}

\begin{lemma}
\label{martbd0}
Let $\{M_n\} \subset \R^d$ be an $\{\SC{F}_n\}$-martingale with $\EE[\|M_0\|^p] < \infty$ and assume that for some  $p>2$, there exists a constant $\nu$ such that
$$\EE[\|M_{n+1}-M_{n}\|^p|\SC{F}_n] \leq \nu \quad \text{for all } n \geq 0.$$
Let $\tau=\inf\{n>0:\|M_n\| <  n\}$. 
Then for $0<r<p$, there exists a constant $\theta\Let\theta(\EE[\|M_0\|^p],\nu,p,r)$ such that
$$ \EE\bigl[\|M_n\|^r\indic{\{\tau> n\}}\bigr] \leq \theta/n^{p-r}.$$
\end{lemma}
The proof follows by combining Lemma \ref{bdS} and Lemma \ref{bdS2}. The steps are essentially similar to that of \citep[Theorem 4]{ref:PemRos-99}. However to make our presentation clear, we felt the need to fill in  the necessary details for our case (see Appendix).

\begin{lemma}
\label{supmartbd}
Let $\{Z_n\}$ be an $\{\SC{F}_n\}$-adapted  process taking values in $\R^d$, $\Omega_0 \subset \Omega$ measurable, $D_n$ a sequence of measurable subsets of $\R^d_+$ and $\sigma$ a stopping time. Call $\Delta_n = Z_{n+1} - Z_n$. Let $\{\gamma_n\}$ be an  $\{\SC{F}_n\}$-adapted $\R^d$-valued process . 
Suppose that 
\begin{itemize}[label=$\circ$, leftmargin=*]
\item $\{Z_n\}$ is a supermartingale for all $1\leq n \leq \sigma$, that is, $\EE[\Delta_n|\SC{F}_n] \leq 0$ for $1\leq n <\sigma$,
\item $E[\|Z_0\|^p] < \infty$'
\end{itemize}
and either
\begin{enumerate}[label={\rm (\roman*)}, leftmargin=*, align=right]
\item \label{i}$E[\|\Delta_n\|^p|\SC{F}_n] \leq L$, for  $0\leq n <\infty$;
\end{enumerate}
or
\begin{enumerate}[label={\rm (\roman*)}, leftmargin=*, align=right, start=2]
\item \label{ii}$E[\|\Delta_n\|^p] \leq L$, for  $0\leq n <\infty$.
\end{enumerate}
Assume that  on  $\Omega_0\cap\{n <\sigma\}$
\begin{enumerate}[label={\rm (\roman*)}, leftmargin=*, align=right,start =3]
 \item\label{sup1} $Z_k - \gamma_k \in D_k$, for $1\leq k\leq n$;
 \item\label{sup2} there exists a constant $\beta>0$ such that $\gamma_k \geq 0$ and $\|\gamma_k\| \geq k\beta$ for $1\leq k \leq n$.
 \end{enumerate}
Then for any $0<r<p$, there exists a constant $\theta \Let \theta(E[\|Z_0\|^p],L,p,r,\beta)$ such that in case of \ref{i}
$$\EE\bigl[\norm{Z_n}^r\indic{\Omega_0\cap\{n<\sigma \}}\bigr] \leq \frac{\theta}{n^{p-r}},$$
while in the case of \ref{ii}
$$\EE\bigl[\norm{Z_n}^r\indic{\Omega_0\cap\{n<\sigma \}}\bigr] \leq \frac{\theta}{n^{p/2-r}},$$
\end{lemma}

\begin{proof} As in the proof of \citep[Corollary 5]{ref:PemRos-99}, the proof relies on a clever use of Doob's decomposition \citep[Theorem 5.2.10]{ref:Dur-10}, \citep[p.\ 74]{ref:EthKur-86}. By Doob's decomposition on each component, there exists a (component-wise) increasing predictable process $\{V_n\}_{n\geq 1}$ and a martingale $\{M_n\}_{n\geq 1}$ with $M_1 = Z_1$ such that 
$$Z^{\sigma}_n\equiv Z_{\sigma\wedge n} = M_n - V_n, \quad n\geq 1.$$
Note that  since $\{V_n\}$ is predictable, $\EE[Z^{\sigma}_{n+1} -Z^{\sigma}_n|\SC{F}_n ] = -(V_{n+1}-V_n).$ Hence, \\
$$\norm{V_{n+1}-V_n}^p \leq \EE\bigl[\norm{\Delta_n}^p \,\big|\, \SC{F}_n\bigr] \leq L.$$
Therefore,
\begin{align*}
\EE\bigl[\norm{M_{n+1}-M_n}^p\,\big|\,\SC{F}_n\bigr] & \leq 2^p \bigl( \EE\bigl[\norm{\Delta_n}^p\big|\, \sigalg_n\bigr] + \|V_{n+1}-V_n\|^p\bigr) \\
& \leq 2^{p+1}L.
\end{align*}
Next, observe that on $\Omega_0\cap\{n <\sigma\}$,
$ Z_{k } -\g_{k} \in D_k \subset \R^d_+, \ \ 1\leq k\leq n.$\\
Hence by \ref{sup1} and  \ref{sup2}, on $\Omega_0\cap\{n <\sigma\}$, $Z_{k} \geq \gamma_{k} \geq 0$, implying $M_{k} \geq V_{k } +\g_{k}\geq \g_{k}\geq 0,$ for $k\leq n$.
Observe that for $x,y\in \R^{d}$, $x\geq y\geq 0$ implies $\|x\| \geq \|y\|$.
It follows from \ref{sup2}  that 
$$\Omega_0\cap\{n<\sigma \}\subset \{\|M_{k}\| \geq k\beta, k\leq n\} = \{\tau>n\},$$ where
$\tau =\inf\{n> 0: \|M_n/\beta\| < n\}$.
Moreover, since $M_n = Z_n + V_n \geq Z_n$ and on $\Omega_0\cap\{n <\sigma\}$,  $Z_{n} \geq 0$ 
$$\|Z_n\|^r\indic{\Omega_0\cap\{n<\sigma \}}\leq \|M_n\|^r \indic{\{\tau> n\}}.$$
Now putting $M_0 = Z_0$ and using the fact that $E[\|Z_0\|^p] <\infty$, we have $E[\|M_1\|^p] \leq 2^p(\EE[\|Z_1-Z_0\|^p] + E[\|Z_0\|^p] )\leq 2^p(L+ E[\|Z_0\|^p)< \infty$.
The assertion now follows by applying Lemma \ref{martbd0} to the martingale $\{M_{n}/\beta\}_{n\geq 1}$.

The steps are almost exactly the same if we have \ref{ii} instead of \ref{i}, except now we apply Lemma \ref{martbd} to $\{M_{n}/\beta\}$.
\end{proof}

\newcommand{\exittime}{\tau_e}
\begin{proof}[Proof of Theorem \ref{mainth}]
Fix $N \geq 1$. Notice that 
$$\EE[\|G_N(X_N)\|^r] = \EE[\|G_N(X_N)1_{\{X_N \in C\}}\|^r] + \EE[\|G_N(X_N)1_{\{X_N \in \metsp\setmin C\}}\|^r].$$
Since by \ref{condg3}, $\EE[\|G_N(X_N)1_{\{X_N \in C\}}\|^r] <b^r$, we need to prove that \\ $\sup_{n} \EE[\|G_n(X_n)1_{\{X_n \in \metsp\setmin C\}}\|^r] <\infty$.

To this end, define the last  exit time $\exittime$ of the process $\{X_n\}$ from $C$ up to time $N$ by
$$\exittime = \max\{k\leq N \mid X_k \in C\}.$$
Note that 
\begin{equation}
\label{expdec}
\begin{aligned}
\EE\Bigl[\norm{G_N(X_N)\indic{\{X_N \in \metsp\setmin C\}}}^r\Bigr] & = \sum_{k=0}^{N}\EE\Bigl[\norm{G_N(X_N)\indic{\{X_N \in \metsp\setmin C\}}}^r \indic{\{\exittime=k\}}\Bigr]\\
	& = \sum_{k=0}^{N-1}\EE\bigl[\norm{G_N(X_N)}^r \indic{\{\exittime=k\}}\bigr].
\end{aligned}
\end{equation}
For any $k<N$, define the random variables $\g^{(k)}_n$, by
\[
\g^{(k)}_n =\g^{(k)}_{1} + \sum_{j=1}^{n-1}H_{k+j}(X_{k+j}), \qquad n\geq 2,
\]
with $\g^{k}_0 = 0$ and $\g^{k}_{1} =(a,0,\hdots,0)$.  Define the process $\bigl\{Z^{(k)}_n\bigr\}_{n}$ by
$$Z^{(k)}_n = (G_{k+n}(X_{k+n}) + \g^{(k)}_n)1_{\{X_k \in C\}}.$$
Notice that on the event $\{\exittime =k\}$, $X_k \in  C$ and $X_{k+n} \in \metsp\setmin C$ for all $1 \leq n \leq N-k$. Hence by the assumptions on the sequences $\{G_n\}$ and $\{H_n\}$, on the event $\{\exittime =k\}$
\begin{align}
\label{1}
Z^{(k)}_n-\g^{(k)}_n& =G_{k+n}(X_{k+n}) \in  G_{k+n}(\metsp\setmin C)\subset \R^d_+, \qquad 1\leq n \leq N-k.\\
\label{11}
\g^{(k)}_n&\geq 0.
\end{align}
One consequence of the above observation is that on  $\{\exittime =k\}$
\begin{align}
\label{2}
Z^{(k)}_n \geq G_{k+n}(X_{k+n}) \geq 0, \text{ hence} \quad \norm{Z^{(k)}_n} \geq \norm{G_{k+n}(X_{k+n})} \text{ for }  1\leq n \leq N-k.
\end{align}

Define the stopping time 
 $$\s^{(k)} \Let \inf \Bigl\{j>0 \,\Big|\, Z^{(k)}_j  -\g^{(k)}_j \in \R^d\setmin G_{k+j}(\metsp\setmin C)\Bigr\}.$$
 It is immediately clear from \eqref{1} that
\begin{align}
\label{containment}
\{\exittime=k\} \subset \{\s^{(k)} > N-k\}.
\end{align}

\newpage
{\bf Claim:}\begin{itemize}
\item $\EE[Z^{(k)}_{n+1} -Z^{(k)}_n|\SC{F}^{(k)}_n] \leq 0$, for $1\leq n<\sigma^{(k)}$, where $\SC{F}^{(k)}_n = \SC{F}_{k+n}$;
\item $\sup_{k,n\geq 0}\EE\bigl[\bigl\|Z^{(k)}_{n+1} -Z^{(k)}_n\bigr\|^p \,\big|\,\SC{F}^{(k)}_n\bigr] < \infty$;
\item $\sup_k \EE[\|Z^{(k)}_0\|^p] < \infty$.
\end{itemize}
{\bf Proof of Claim:} Suppose $k$ is such that  $X_k \notin C$. Then from the definition, $Z^{(k)} \equiv 0$ and the assertions in the claim are trivially satisfied. Next, suppose that $k$ is such that $X_k \in C$. Then 
$Z^{(k)}_n = G_{k+n}(X_{k+n}) + \g^{(k)}_n.$
Observe that for $1\leq n< \s^{(k)}$, we have $Z^{(k)}_{j} - \g^{(k)}_{j} =G_{k+j}(X_{k+j})\in G_{k+j}( \metsp\setmin C)$ for all $1\leq j \leq n$. 
It follows from \ref{condg2} that $X_{k+n} \in \metsp\setmin C$, for $1\leq n <\sigma^{(k)}$ and we have using \eqref{cond1} 
\begin{align*}
\EE\bigl[Z^{(k)}_{n+1} -Z^{(k)}_n \,\big|\,\SC{F}^{(k)}_n\bigr] &= \EE\bigl[G_{k+n+1}(X_{k+n+1}) - G_{k+n}(X_{k+n})\,\big|\,\SC{F}^{(k)}_n\bigr] + H_{k+n}(X_{k+n}) \\
&\leq 0.
\end{align*}
where $\SC{F}^{(k)}_n = \SC{F}_{k+n}$.
Moreover for $0\leq n <\infty$,
\begin{align*}
 \EE\bigl[\bigl\|Z^{(k)}_{n+1} -Z^{(k)}_n\bigr\|^p \,\big|\,\SC{F}^{(k)}_n\bigr] &\leq 2^p\left(\EE\bigl[\|G_{k+n+1}(X_{k+n+1}) - G_{k+n}(X_{k+n})\|^p\,\big|\,\sigalg^{(k)}_n\bigr] + b^p\right) \\
 & \leq 2^p(L+b^p).
\end{align*}
Lastly, by \ref{condg3}, $\EE[\|Z^{(k)}_0\|^p] = \EE[\|G_k(X_k)\|1_{\{X_k \in C\}}] < \sup_n\sup_{x\in C} \|G_n(x)\|^p \leq b^p$.
Hence the claim follows.
\vspace{.4cm}

Furthermore, on $\{\exittime =k\} \cap\{n<\sigma^{(k)}\}$, $Z^{(k)}_{j} - \g^{(k)}_{j} \in G_{k+j}( \metsp\setmin C) $ and $X_{k+j} \in \metsp\setmin C$ for all $1\leq j\leq n$ implying that  $\g^{(k)}_j \geq 0$ for $1\leq j\leq {n+1}$. 
Observe that that from the definition of $\gamma^{(k)}_{n}$, $\|\gamma^{(k)}_{n}\| \geq n\beta$ for some constant $\beta$. 
To see this use  the fact that there exist constants $\beta_{1}$ and $\beta_{2}$ such that $\beta_{1}\|x\| \leq \|x\|_{1}\leq \beta_{2}\|x\|$ for all $x\in\R^{d}$, where $\norm{\cdot}_{1}$ denotes the standard $\ell_{1}$-norm on $\R^{d}$.
Next noting that $\|x+y\|_{1} =\|x\|_{1} +\|y\|_{1}$ for $x,y\in \R_+^{d}$ we have
\begin{align*}
\|\gamma^{(k)}_{n}\| \geq \frac{1}{\beta_{2}}\|\gamma^{(k)}_{n}\|_{1}& =\frac{1}{\beta_{2}} \left(\|\gamma^{(k)}_{1}\|_{1} + \sum_{j=1}^{n-1}\|H_{k+j}(X_{k+j})\|_1\right) \\
&\geq \frac{\beta_{1}}{\beta_{2}} \left(\|\gamma^{(k)}_{1}\| + \sum_{j=1}^{n-1}\|H_{k+j}(X_{k+j})\|\right)\\
&\geq \frac{\beta_{1}}{\beta_{2}}\left(a+a(n-1)\right)\\
& \geq\beta n, \mbox{ for some constant } \beta \mbox{ depending on }a\mbox{ and dimension } d.
\end{align*}
Now Lemma \ref{supmartbd} gives that there exists a constant $\theta$ (depending on $p,a,b,d,L,r$) such that for $0 < r<p$,
\begin{align}
\label{subbd}
\EE\Bigl[\|Z^{(k)}_{N-k}\|^{r}\indic{\{\exittime =k\} \cap\{\sigma^{(k)}> N-k\}}\Bigr] \leq \frac{\theta}{(N-k)^{p-r}}.
\end{align}
Finally, from \eqref{2} and \eqref{containment} it follows that for $k=0,\hdots,N-1$,
$$\|G_N(X_N)\|^r\indic{\{\exittime = k\}} = \|G_{k+N-k}(X_{k+(N-k)})\|^r\indic{\{\exittime = k\}} \leq \|Z^{(k)}_{N-k}\|^r\indic{\{\exittime =k\} \cap\{\s^{(k)}>N-k\}}.$$
Hence by Lemma \ref{supmartbd},  we have
\begin{align*}
\EE\bigl[\|G_N(X_N)\indic{\{X_N\in \metsp\setmin C\}}\|^r\bigr] &= \sum_{k=0}^{N-1}\EE\bigl[\|G_N(X_N)\|^r\indic{\{\exittime = k\}}\bigr] \leq  \sum_{k=0}^{N-1}\frac{\theta}{(N-k)^{p-r}}\\
& \leq \theta\zeta(p-r) <\infty, \quad \mbox{ for } 0 < r<p-1.
\end{align*}
where $\zeta$ denotes the Riemann zeta function.

If in \ref{2condition}, instead of \eqref{cond2} we have \eqref{cond2_ii}, then by Lemma  \ref{supmartbd}  instead of \eqref{subbd}, we have
\begin{align*}
\EE\Bigl[\|Z^{(k)}_{N-k}\|^{r}\indic{\{\exittime=k\}\cap\{\sigma^{(k)}> N-k\}}\Bigr] \leq \frac{\theta}{(N-k)^{p/2-r}},
\end{align*}
and the rest of the proof stays the same.
\end{proof}

Let $\{\ort^d_\alpha \mid \alpha = 1, \ldots, 2^d\}$ denote the collection of all standard orthants of $\R^d$, i.e., the sets $\{z\in\R^d\mid z_i \ge 0 \text{ or }z_i \le 0\text{ for each }i=1, \ldots, d\}$. Recall that if $\ort$ is a non-empty positive convex cone in $\R^d$, then the conic (partial) order $\le_{\ort}$ induced by $\ort$ is defined by $x \le_{\ort} y$ if  $y - x \in \ort$. For $x,y \in \R^d$ and an orthant $\ort^d_\alpha$ we define $x\leq_\alpha y$ if $y-x \in \ort^d_\alpha$. To keep consistency with our earlier notation, when the orthant in consideration is $\R^d_+$, we write $x\leq y$, whenever $y-x\in \R^d_+$.  We note that in Theorem \ref{mainth}, the orthant $\R^d_+$ plays no special role, that is, the statement of the theorem is true if we replace $\R^d_+$ by any orthant $\ort^d_\alpha$. More precisely,
\begin{theorem}
\label{t:main_orth}
			Let $(\Omega, \sigalg, \{\sigalg_n\}, \PP)$ be a filtered probability space, $\metsp$ a complete and separable metric space and $C\subset \metsp$ measurable.
			Let the orthant $\ort^d_\alpha$ be defined by
			$$\ort^d_\alpha= \{x \in \R^d \mid x_{i_1}\leq 0,\hdots x_{i_l}\leq 0, x_j \geq 0, j\neq i_1,\hdots, i_l\}.$$
			  Let $\{G_n:\metsp \rt \R^d\}$ and $\{H_n:\metsp \rt \R^d\}$ be  sequences of measurable functions satisfying
			\begin{enumerate}[label={\rm (\roman*)}, leftmargin=*, align=right]
			\item \label{orthcondg1}for every $n$, $G_n, H_n : \metsp\setmin C \rt \ort^d_\alpha$;
			\item \label{orthcondg2}for every $n$, $G_n^{-1}G_n( \metsp\setmin C) = \metsp\setmin C$;
			\item \label{orthcondg3} there exist constants $a,b>0$, such that $a\leq\inf_n\inf_{x\notin C}\|H_n(x)\|, \\ \sup_n\sup_{x\notin C}\|H_n(x)\|\leq b$ and $\sup_n\sup_{x\in C}\|G_n(x)\| \leq b$.
			\end{enumerate}
			 Let $\{X_n\}$ be a sequence of $\{\SC{F}_n\}$-adapted  $\metsp$-valued random variables. Assume that $X_0\in C$ and the following two conditions hold:
			\begin{enumerate}[label={\rm (\roman*)}, leftmargin=*, align=right,start=4]
						\item for all $n\geq0$,
					\begin{equation}
					\label{orthcond1}
						\EE\bigl[G_{n+1}(X_{n+1})-G_n(X_n)\,\big|\,\SC{F}_n\bigr] \leq_\alpha - H_n(X_n) \quad \text{on}\quad \{X_n \notin C\};
					\end{equation}
				\item there exist constants $L > 0$ and $p>2$ such that for all $n\geq0$
					\begin{align}
					\label{orthcond2}
						\EE\bigl[\|G_{n+1}(X_{n+1})-G_n(X_n)\|^p\,\big|\,\SC{F}_n\bigr] \leq L.
					\end{align}
			\end{enumerate}
			Then for any $0<r<p-1$, there exists a constant $\eta \Let \eta(p,a,b,d,L,r)$ such that 
			\[
				\sup_{n\in\Nz} \EE\bigl[\|G_n(X_n)\|^r \bigr] \leq \eta.
			\]
			If  instead of \eqref{orthcond2} we have
			\begin{align}
					\label{orthcond2_ii}
						\EE\bigl[\|G_{n+1}(X_{n+1})-G_n(X_n)\|^p\bigr] \leq L, 
					\end{align}
		        then  for any $0<r<p/2-1$, there exists a constant $\eta \Let \eta(p,a,b,d,L,r)$ such that 
			\[
				\sup_{n\in\Nz} \EE\bigl[\|G_n(X_n)\|^r \bigr] \leq \eta.
			\]

		\end{theorem}
		
		\begin{proof}
		The proof is exactly the same as that of Theorem \ref{t:main} since all the steps remain valid if we replace $\R^{d}_+$ by $\ort^d_\alpha$. Alternatively, we can derive Theorem \ref{t:main_orth} from Theorem \ref{t:main} by `rotating' the orthant $\ort^d_\alpha$ to $\R^d_+$. \\	
		For $x\in \R^d$, define the operator $\delta_\alpha $ on $\R^d$ by the following action on $x$:
\[
 	\text{\emph{alter the sign of $x_{i_1},\hdots,x_{i_l}$, and keep the remaining co-ordinates unchanged.}}
\]
	Note that $\delta_\alpha$ is a self adjoint and unitary operator and $\delta_\alpha(\ort^d_\alpha) = \R^d_+$.
	Define the sequences  $\{G^\alpha_n\}$ and $\{H^\alpha_n\}$ by
	$$G^\alpha_n(x) = \delta^\alpha G_n( x), \ \ H^\alpha_n(x) = \delta^\alpha H_n( x) $$
	Then \ref{orthcondg1} and \ref{orthcondg2} imply \ref{condg1} and \ref{condg2} of Theorem \ref{t:main} for the sequences $\{G^\alpha_n\}$ and $\{H^\alpha_n\}$. Also, $a\leq\inf_n\inf_{x\notin C}\|H_n^\alpha(x)\|, \sup_n\sup_{x\notin C}\|H_n^\alpha(x)\|\leq b$ and $\sup_n\sup_{x\in C}\|G_n^\alpha(x)\|\leq b.$
	Finally, \eqref{orthcond1} and \eqref{orthcond2} imply that the sequence of processes $\{ X_n\}$ satisfies
	\begin{itemize}
				\item $\displaystyle{	\EE\bigl[G^\alpha_{n+1}(X_{n+1})-G_n^\alpha( X_n)\,\big|\,\SC{F}_n\bigr] \leq - H^\alpha_n( X_n) \quad \text{on}\quad \{ X_n \notin C\}},$
					
				\item $\displaystyle{\EE\bigl[\|G^\alpha_{n+1}( X_{n+1})-G^\alpha_n( X_n)\|^p\,\big|\,\SC{F}_n\bigr] \leq L}.$
				
		\end{itemize}
		Consequently, Theorem \ref{t:main} says that there exists an $\eta>0$ such that for all $n \in \Nz$,
		$$\EE\left[\|G^\alpha_n( X_n)\|^r \right]  =\EE\left[\|G_n( X_n)\|^r \right] \leq \eta$$

\end{proof}

\subsection*{{\bf Existence of invariant probability measures for Markov chains}}
	\newcommand{\proc}{X}
	\def\tp{P}

		Let $\metsp$ be a complete and separable metric space. Let $\{\proc_n\}$ be a $\metsp$-valued Markov process with transition kernel $\tp : \metsp\times\Borelsigalg{\metsp}\lra[0, 1]$, where $\Borelsigalg{\metsp}$ denotes the Borel $\sigma$-algebra on $\metsp$.  For $g:\metsp\rt\R^d$, define $\tp g:\metsp \rt \R^d$ by $\tp g(x) \Let \int_{\metsp} \tp(x, \drv y) g(y)=\EE[g(X_1)|X_0 = x] =\EE_x[g(X_1)]$. Suppose that $C\subset \metsp$ is measurable and $\bar{\Ball}_\kappa$ denotes the closed Euclidean ball of radius $\kappa$ centered at the origin in $\R^d$. Assume that 
		\begin{enumerate}[label={\rm (\roman*)}, leftmargin=*, align=right]
		\item $\tp$ is \emph{(weak) Feller}, i.e., if $f:\metsp \lra \R$ is a continuous and bounded function, then $\tp f$ is continuous and bounded;
		\item  there exist a measurable map $G:\metsp\rt\R^d$, a measurable function $H:\metsp\rt\R^d$ and  constants $a,b>0$ such that
		\begin{itemize}[label=$\circ$, leftmargin=*]
		        \item  there exists  $\alpha\in\{1,\hdots,2^d\}$ such that $G, H : \metsp \setmin C \rt \ort^d_{\alpha}$;
			\item  $G^{-1}G(\metsp \setmin C) = \metsp \setmin C$;
			\item $G^{-1}(\bar{B}_\kappa) \equiv\{x\in \R^m:\|G(x)\| \leq \kappa\}$ is compact for every $\kappa>0$;
			\item  $\sup_{x\in C} \|G(x)\| <\infty$;
			\item $a\leq \|H(x)\|\leq b$, for all $x\notin C$;
		\end{itemize}
	\item	 for all $x\in \metsp \setmin C$
		\begin{equation}
		\label{e:vecdriftMarkov}
			\tp G(x) - G(x) \leq_{\alpha} -H(x);
		\end{equation}
		
		\item there exist constants $L>0$ and $p>2$ such that for all $x\in\metsp$
		\[
			\tp( \|G(\cdot) -G(x)\|^p)(x) =  \EE_x\bigl[\norm{G(\proc_1) - G(x)}^p \bigr] \le L.
		\]
		\end{enumerate}
Then $\{X_n\}$ has an invariant probability measure. To see this, first observe that an application of  Theorem \ref{t:main_orth} gives  $\sup_n\EE_x\bigl[\|G(X_n)\|^r\bigr] < \infty$ for all $0<r<p-1$ and $x\in C$. Fix  $0<r<p-1$ and  $x\in C$.
Let $\epsilon >0$ and let $\kappa$ be such that $\sup_n\EE_x\bigl[\|G(X_n)\|^r\bigr]/\kappa<\epsilon.$
Notice that
\begin{align*}
\PP_x(X_n \notin G^{-1}(\bar{B}_\kappa)) =\PP_x (\|G(X_n)\| > \kappa) \leq \sup_n\EE_x\bigl[\|G(X_n)\|^r\bigr]/\kappa <\epsilon.
\end{align*}
Since $G^{-1}(\bar{B}_\kappa)$ is compact, it follows that   $\{\tp^n(x,\cdot)\}$ is tight.
Define the C\`esaro sum $\tp^{(n)}$ by
$$\tp^{(n)}(x,\cdot)  \Let \frac{1}{n}\sum_{k=0}^{n-1}\tp^k(x,\cdot).$$
It is immediately clear that the sequence of probability measures $\{\tp^{(n)}(x,\cdot)\}$ is tight and hence relatively compact. Let $\mu$ be a probability measure on $\metsp$ which is a limit point of  $\{\tp^{(n)}(x,\cdot)\}$. Then the Krylov-Bogoliubov theorem (\citep[Proposition 7.2.2]{ref:Her-LerLas-03}, \citep[Theorem 3.1.1]{ref:daPrato96}) shows that $\mu$ is invariant.

	\section{Uniform moment bounds for discrete-time iterated function systems}
	\label{s:app3}
		Consider a discrete-time Markov process $\{Z_t\}_{t\in\Nz} \Let \{(x_t, y_t)\}_{t\in\Nz}$ taking values in $\R^d_+\times \PSet$, where $\PSet$ is a countable set, defined by the following rules:
		\begin{enumerate}[label=(IFS\arabic*), align=left, leftmargin=*]
			\item \label{cond:ifs:maps} for each $i\in\PSet$ there exists a measurable mapping $f(\cdot, i):\R^d_+ \lra \R^d_+$;
			\item \label{cond:ifs:trk} there exists a measurable map $P:\R^d_+ \times \PSet\times \PSet\lra[0, 1]$ such that for each $(x, y)\in\R^d_+\times\PSet$ the function $P(x, y, \cdot)$ is a transition probability;
			\item \label{cond:ifs:transition} at time $t = n$, given the state $(x_n, y_n) = (x, y)$, 
				\begin{itemize}[label=$\circ$, leftmargin=*]
					\item first $y_{n+1}$ is selected randomly according to a time-homogenous but $x$-dependent transition kernel $P_x(y, \cdot) \Let P(x, y, \cdot)$, and
					\item given $y_{n+1}$, we set $x_{n+1} = f(x, y_{n+1})$.
				\end{itemize}
		\end{enumerate}
		Observe that neither of the process $\{x_t\}_{t\in\Nz}$ or $\{y_t\}_{t\in\Nz}$ is Markovian on its own. The transition kernel of the process $\{Z_t\}_{t\in\Nz}$ stands as
		\[
			\PP\bigl(Z_{t+1} = (x', y') \mid Z_t = (x, y)\bigr) = P_x(y, y') \delta_{f(x, y')}(x'),
		\]
		where $\delta$ is the Dirac measure. The stochastic system
		\begin{equation}
		\label{e:ifs}
			x_{t+1} = f(x_t, y_{t+1}),\qquad (x_0, y_0) \in\R^d_+\times\PSet \text{ given},
		\end{equation}
		derived from the process $\{Z_t\}_{t\in\Nz} = \{(x_t, y_t)\}_{t\in\Nz}$ constructed above is known as an \emph{iterated function system with place dependent probabilities} \citep{ref:BarDemEltGer-88}. These systems are generally employed in the synthesis of fractals, modeling biological phenomena, etc \citep{ref:LasMac-94}.

		Iterated function systems are important objects of study in control theory, where they are known by the name discrete-time stochastic hybrid systems \citep{ref:ChaPal-11, ref:AbaKatLygPra-10, ref:SumLyg-10, ref:ChaCinLyg-11}. There is a considerable literature addressing classical weak stability questions concerning the existence and uniqueness of invariant measures of iterated function systems, see e.g., \citep{ref:Pei-93, ref:LasYor-94, ref:Sza-03, ref:DiaFre-99, ref:JarTwe-01} and the references therein. The arguments in these articles predominantly revolve around average contractivity conditions of the iterated function system, and continuity of the probability transitions. Stronger stability notions such as existence of moments of sufficiently high order mostly involve Foster-Lyapunov drift conditions, which in turn work best under the average contractivity assumption. Although there have been efforts to relax average contractivity conditions in conjunction with Foster-Lyapunov drift conditions, see e.g., \citep{ref:DouForMouSou-04}, generally the assertions consist of sub-geometric rates of convergence of Markov processes to their invariant measures; moreover, such techniques do not extend directly to moment bounds. Furthermore, in real-world control applications the average contractivity property generally translates to requiring unbounded control actions, which is hardly ever possible to guarantee. In this section we give conditions for uniform $\Lp r(\PP)$-boundedness of the process $\{x_t\}_{t\in\Nz}$ generated by \eqref{e:ifs} in the absence of average contractivity. To this end, we further assume that
		\begin{enumerate}[label=(IFS\arabic*), leftmargin=*, align=left, start=4]
			\item \label{cond:ifs:boundedjump} there exists a constant $L > 0$ such that 
				\[
					\norm{x - f(x, i)} \le L\quad \text{for every }(x, i)\in\R^d_+\times\PSet.
				\]
		\end{enumerate}
		\begin{remark}
			Observe that the existence of a uniform bound on the jumps hypothesized in condition \ref{cond:ifs:boundedjump} above implies that an ``average contractivity'' condition is impossible to satisfy without transforming coordinates. The condition \ref{cond:ifs:boundedjump} holds for a large class of realistic nonlinear control systems, especially under bounded control actions.
		\end{remark}

		We have the following:
		\begin{proposition}
		\label{p:ifs}
			Consider the system \eqref{e:ifs} and suppose that the conditions \ref{cond:ifs:maps}, \ref{cond:ifs:trk}, \ref{cond:ifs:transition}, and \ref{cond:ifs:boundedjump} hold. In addition, suppose that there exist a measurable bounded set $C\subset\R^d_+$ and a vector $a\in\R$ with $a > 0$ such that
			\[
				\sum_{y'\in\PSet} P_x(y, y') f(x, y') - x \le - a\frac{x}{\norm{x}}\quad\text{for }(x, y)\in(\R^d_+\setmin C)\times\PSet.
			\]
			Then the process $\{x_t\}_{t\in\Nz}$ is $\Lp r(\PP)$-bounded for every $r > 0$.
		\end{proposition}
		\begin{proof}
			Let $\{\sigalg_t\}_{t\in\Nz}$ be the natural filtration generated by the process $\{Z_t\}_{t\in\Nz}$. For any given $p > 2$, we see at once that the condition \ref{cond:ifs:boundedjump} implies that 
			\begin{equation}
			\label{e:Lpbound}
				\EE\bigl[\norm{x_{t+1} - x_t}^p\,\big|\,\sigalg_t\bigr] \le L^p\quad \text{for all }t\in\Nz;
			\end{equation}
			therefore, condition \ref{cond2} of Theorem \ref{t:main} holds. Moreover,
			\begin{align*}
				\EE\bigl[x_{t+1} - x_t\,\big|\,\sigalg_t\bigr] & = \int_{\R^d_+} \sum_{y'\in\PSet} P_{x_t}(y_t, y') \delta_{f(x_t, y')}(x')\id(x')\,\drv x' - x_t\\
					& = \sum_{y'\in\PSet} P_{x_t}(y_t, y') f(x_t, y') - x_t\\
					& \le -a \frac{x_t}{\norm{x_t}}\quad\text{on }\{x_t\in \R^d_+\setmin C\}
			\end{align*}
			in view of our hypotheses; therefore, condition \ref{cond1} of Theorem \ref{t:main} holds with $G_n(x) \equiv x$ and $H_n(x) = a\frac{x}{\|x\|}$. We conclude that the process $\{x_t\}_{t\in\Nz}$ is $\Lp r(\PP)$ bounded for all $0 < r < p-1$. Furthermore, since the bound on the right-hand side of \eqref{e:Lpbound} is finite for all $p > 0$, Theorem \ref{t:main} also implies that $\{x_t\}_{t\in\Nz}$ is $\Lp r(\PP)$-bounded for every $r > 0$. The assertion follows.
		\end{proof}

	\section{Connection to biochemical reaction systems}
	\label{s:app4}
\def\Nsp{n}
\def\Nreact{v}
\def\change{\nu}

A biochemical reaction system involves multiple chemical reactions and several species. In general, chemical reactions in single cells occur far from thermodynamic equilibrium and the number of molecules of chemical species is often low \citep{ref:Kei87, ref:Gup95}. Recent advances in real-time single cell imaging, micro-fluidic techniques and synthetic biology have testified to the random nature of gene expression and protein abundance in single cells \citep{ref:Yuetal06, ref:Friedman10}. Thus a stochastic description of chemical reactions is often mandatory to analyze the behavior of the system.  The dynamics of the system is typically modeled by a continuous-time Markov chain (CTMC) with the state being the number of molecules of each species. \citep{ref:AndKur-11} is a good reference for a review of the tools of Markov processes used in the reaction network systems. Analyzing stability of stochastically modeled biochemical reaction systems (e.g, gene regulatory networks) in particular, questions dealing with existence of invariant probability measures, moment bounds are important both for experimental and theoretical purpose \citep{ref:AndCraKurtz10, ref:SamKham04}. The goal of this section is to outline a method to investigate these kind of stability questions for biochemical reaction networks.

 Consider a biochemical reaction system consisting of $\Nsp$ species and $\Nreact$ reactions, and let $X(t)$  denote the state of the system at time $t$ in $\Z^\Nsp_+$. If the $k$-th reaction occurs at time $t$, then the system is updated as 
           $X(t) = X(t-) + \change^+_{k}-\change^-_{k},$
where $X(t-)$ denotes the state of the system just before time $t$, and $\change^-_{k}, \change^+_{k} \in \Z^\Nsp_+$ represent the vector of number of molecules consumed and created in one occurrence of reaction $k$, respectively. For convenience, let $\change_k \Let  \change^+_{k}-\change^-_{k}$. The evolution of the process $X$ is modeled by
$$\PP[X(t+\Delta t) = x +\change_k| X(t) =x] = a_k(x)\Delta t + o(\Delta t).$$
The quantity $a_k$ is usually called the {\em propensity} of the reaction $k$ in the chemical literature, and its expression is often calculated by using the {\em law of mass action} \citep{ref:wilkinson_2006, ref:Gillespie2007}. The generator matrix or the $Q$-matrix of the CTMC $X$  is given by
$q_{x,x+\change_k} =a_k(x).$
The CTMC $X$ will have an {\em invariant measure} $\pi$ if $\pi Q\equiv 0$.

Let $\Ball_\rho$ be the standard open ball of radius $\rho$ centered at $0$ in $\R^\Nsp$.  Assume that there exist a function $H:\Z^\Nsp_+\rt \R^\Nsp_+$ and a constant $\rho>0$ such that 
\begin{enumerate}[label={\rm (BRS\arabic*)}, leftmargin=*, align=right]
         \item \label{RS1}there exist constants $a, b>0$ such that $a\leq \|H(x)\| \leq b$, for $x \in \Z^\Nsp_+ \setmin \Ball_\rho$;
	\item \label{drift} $\displaystyle{F(x) \Let \sum_{k=1}^\Nreact a_k(x) \change_k \leq - H(x)A(x)\quad\text{for }x\in \Z^\Nsp_+\setmin  \Ball_\rho }$, where  $A(x)  =\sum_{k=1}^\Nreact a_k(x)$;
	\item \label{mmt}$\displaystyle{F_p(x) \Let \sum_{k=1}^\Nreact a_k(x) \|\change_k\|^p \leq L A(x)} $
                                    for some constants $p>2$ and  $L>0$.
\end{enumerate}
Possible examples of $H$ include constant vector with positive entries, $H(x) = (\alpha x+\beta)/\|\alpha x+\beta\|, \alpha>0,\beta\geq 0$, etc.
Let $\{Y_n\}$ be the jump chain corresponding to the CTMC $X$. That is, putting $\tau_0=0$, we define inductively
\[
	\tau_{n+1} \Let \inf\bigl\{t>\tau_n\,\big|\, X(t) \neq X(\tau_n)\bigr\}.
\]
Notice that $\tau_n$ denotes the $n$-th jump time of the CTMC $X$. Define $Y_n = X(\tau_n)$. $\{Y_n\}$ is a discrete-time Markov chain and is often called the jump chain or the skeleton chain corresponding to the CTMC $X$. Now \ref{RS1} and \ref{drift} imply $\sup_n E\bigl[\norm{Y_n}^r\bigr] <\infty $ for $0<r<p-1$.
To see this, we first obtain the transition matrix of the Markov chain $\{Y_n\}$  from the $Q$-matrix of $X$ (see e.g., \cite[p.\ 108]{ref:Nor-98}). Specifically, 
\begin{eqnarray*}
\PP(Y_{n+1}=x+\nu_k|Y_n = x)& = 
\begin{cases}
a_k(x)/A(x),& \quad \mbox{if } A(x) \neq 0\\
0,& \quad \mbox{if } A(x) =0.
\end{cases}
\end{eqnarray*}

\begin{eqnarray*}
\PP(Y_{n+1}=x|Y_n = x)& = 
\begin{cases}
0,& \quad \mbox{if } A(x) \neq 0\\
1,& \quad \mbox{if } A(x) =0.
\end{cases}
\end{eqnarray*}
In most biochemical reaction systems $A(x) >0$ for all $x \in \Z^\Nsp_+$ or at least outside a compact set of $\Z^\Nsp_+$.
Now for all $x \in Z^\Nsp_+ \setmin \Ball_\rho$, by \ref{drift}
\begin{align*}
\EE[Y_{n+1} -Y_n|Y_n =x] =  \sum_{k=1}^\Nreact \change_ka_k(x)/A(x) \leq -H(x)
\end{align*}
Moreover by \ref{mmt}, for all $x\in \Z^\Nsp_+$
\begin{align*}
\EE[\|Y_{n+1} -Y_n\|^p|Y_n =x] =  \sum_{k=1}^\Nreact \|\change_k\|^pa_k(x)/A(x) \leq L
\end{align*}

 Thus \ref{drift} and \ref{mmt} imply \eqref{cond1} and \eqref{cond2} of Theorem \ref{t:main} for the Markov chain $\{Y_n\}$ with $G_n(x) \equiv x$, and consequently, $\sup_n E\bigl[\norm{Y_n}^r\bigr] <\infty $ for $0<r<p-1$.  Now, the discussion after the  proof of Theorem \ref{t:main_orth} shows that $\{Y_n\}$ has an invariant probability measure $\lambda$. Consequently, it follows from \cite[Theorem 3.5.1]{ref:Nor-98} that if $A(x)>0$ for all $x\in Z^\Nsp_+$, then $\pi(x) \equiv \lambda(x)/A(x)$ is an invariant measure for the CTMC $X$. If $\inf_{x\in \Z^\Nsp_+} A(x) >0$, then the CTMC $X$ has an invariant probability measure. Of course, if we are just interested in the existence of an invariant probability measure and \ref{drift}, \ref{mmt} do not hold, then the discussion after the proof of Theorem \ref{t:main_orth} can be employed to look for a suitable $G$.

\setcounter{section}{0}
\setcounter{theorem}{0}
\setcounter{equation}{0}
\appendix
\section*{Appendix}
\renewcommand{\thesection}{A} 
\renewcommand{\theequation}{A.\arabic{equation}}

\begin{lemma}
\label{bdS}
Let $\{M_n\}$ and $\tau$ be as in Lemma \ref{martbd0} and assume that for some  $p>0$, there exists a constant $\nu$ such that
$$\EE[\|M_{n+1}-M_{n}\|^p|\SC{F}_n] \leq \nu\quad \text{for all } n \geq 0.$$
For $k>0$, let $S_k = \inf\{j\geq 0: \|M_j\|\geq k/3\}$.
Let $T_k=\inf\{j\geq 0: \|M_{j+1} - M_j\| \geq k/3\}$. Then there exists a constant $\theta'$ such that
$$ \EE[\|M_n\|^r1_{\{\tau>n\}}1_{\{S_n\leq T_n\}}] \leq\frac{\theta'}{ n^{p-r}}.$$

\end{lemma}

\begin{proof}
For notational convenience, put $S_n\equiv S$ and $T_n\equiv T$. First, notice that from the definitions, $\{\tau>n\} \subset \{\|M_n\| \geq n\geq n/3\} \subset \{S \leq n\}$. 
Also, on the event $\{S \leq T\}$, $\|M_{S} - M_{S-1}\| \leq  n/3$ and since $\|M_{S-1}\| \leq  n/3 $, we have $\|M_{S}\| \leq 2 n/3$ on $\{S\leq T\}$.
Now
\begin{align}
\label{SleqT}
\non\EE[\|M_n\|^r1_{\{\tau>n\}}1_{\{S\leq T\}}]&  =\EE[\|M_n\|^r1_{\{\tau>n\}}1_{\{S\leq n\}}1_{\{S\leq T\}}] \\
\non& =\EE[ \EE[\|M_n\|^r1_{\{\tau>n\}}1_{\{S\leq n\}}1_{\{S\leq T\}}|\SC{F}_{S}]]\\
 & = \EE[[1_{\{S\leq n\}}1_{\{S\leq T\}} \EE[\|M_n\|^r1_{\{\tau>n\}}|\SC{F}_{S}]]
\end{align}
By the optional sampling theorem, Lemma \ref{burkapp} gives
$$\EE[\|M_n -M_{S\wedge n}\|^p|\SC{F}_{S\wedge n}] \leq  c_p\nu (n-S\wedge n)^{p/2}.$$
Therefore, on the event $\{S\leq n\}$,
$$\EE[\|M_n -M_{S}\|^p|\SC{F}_{S}] \leq  c_p\nu n^{p/2}.$$
Now, for a non-negative random variable $Z$ 
\begin{align}
\label{exp_formula}
\EE[Z^r1_{\{Z\geq u\}}|\SC{G}] = u^r\PP(Z\geq u|\SC{G}) + \int_u^\infty ry^{r-1}\PP(Z\geq y|\SC{G}) \drv y,
\end{align}
where $\SC{G}$ is a sub $\sigma$-algebra of $\SC{F}$ and $\Omega' \in \SC{F}$. Hence
\begin{align*}
\EE[\|M_n\|^r1_{\{\tau>n\}}|\SC{F}_{S}]& = \EE[\|M_n\|^r1_{\{\|M_n\| \geq n\}}|\SC{F}_{S}] \\
& =  n^{r}\PP(\{\|M_n\| \geq n\}|\SC{F}_{S}) + \int_{ n}^\infty ry^{r-1}\PP(\{\|M_n\| \geq y\}|\SC{F}_{S}) \ \drv y\\
& \leq  n^{r}\PP(\{\|M_n  - M_{S}\| \geq n/3\}|\SC{F}_{S}) \\
& \hspace{.5cm}+ \int_{ n}^\infty ry^{r-1}\PP(\{\|M_n- M_{S}\| \geq (y-2 n/3)\}|\SC{F}_{S}) \ \drv y, \mbox { on } \{S\leq T\}\\
& \leq 3^r n^{r-p} \EE[\|M_n  - M_{S}\|^p|\SC{F}_{S}] \\
&\hspace{.5cm}+ \int_{ n}^\infty ry^{r-1}\EE[\|M_n  - M_{S}\|^p|\SC{F}_{S}](y-3n/4)^{-p}\drv y, \mbox{ on } \{S\leq T\}\\
&\leq \theta_1  n^{r-p}\EE[\|M_n  - M_{S}\|^p|\SC{F}_{S}], \mbox{ on } \{S\leq T\}\\
&\leq c_p\nu\theta_1  n^{r-p/2}, \mbox{ on  }\{S\leq T\}\cap \{S\leq n\}.
\end{align*}
Thus \eqref{SleqT} implies that
\begin{align}
\label{Snleqn}
\non\EE[\|M_n\|^r1_{\{\tau>n\}}1_{\{S\leq T\}}]& \leq c_p\nu\theta_1  n^{r-p/2}\EE[1_{\{S\leq n\}}1_{\{S\leq T\}} ] \\
&\leq c_p\nu\theta_1  n^{r-p/2}\PP(S\leq n).
\end{align}
Notice that
\begin{align*}
\PP(S\leq n) &\leq \PP(\|M_{S\wedge n}\|\geq  n/3) \leq 3^p\EE\|M_{S\wedge n}\|^p/ n^p\\
&\leq 3^p\EE\|M_{n}\|^p/ n^p,\quad \mbox{ since } \{\|M_n\|^p\} \mbox{ is a submartingale}\\
& \leq 3^p\frac{\theta_0 n^{p/2}}{ n^p}, \quad \mbox{by } \eqref{EMnbound}
\end{align*}
Plugging the bound for $\PP(S\leq n)$ in \eqref{Snleqn} we are done.
\end{proof}

\begin{lemma}
\label{bdS2}
Let $\{M_n\}$ and $\tau$ be as in Lemma \ref{martbd0} and assume that for some  $p>2$, there exists a constant $\nu$ such that
$$\EE[\|M_{n+1}-M_{n}\|^p|\SC{F}_n] \leq \nu\quad \text{for all } n \geq 0.$$
Let $S_k$ and $T_k$ be defined as in Lemma \ref{bdS}. Then there exists a constant $\theta''$ such that
$$ \EE[\|M_n\|^r1_{\{\tau>n\}}1_{\{S_n> T_n\}}] \leq\frac{\theta''}{ n^{p-r}}.$$
\end{lemma}

\begin{proof}
As before, we denote $S_n\equiv S$ and $T_n\equiv T$. Again since $\{\tau>n\} \subset \{S\leq n\}$,
\begin{align*}
\EE[\|M_n\|^r1_{\{\tau>n\}}1_{\{T<S\}}]& =\sum_{k=0}^{n-1}\EE[\|M_n\|^r1_{\{\tau>n\}}1_{\{S>k\}}1_{\{T=k\}}]\\
&\leq 3^r\sum_{k=0}^{n-1}\EE[(\|M_k\|^r+\|M_{k+1}-M_k\|^r\\
& \quad +\|M_n - M_{k+1}\|^r)1_{\{\tau>n\}}1_{\{S>k\}}1_{\{T=k\}}]\\
&\leq 3^r\sum_{k=0}^{n-1}(I + II + III).
\end{align*}

First notice that since on $\{S>k\}$, $\|M_k\| \leq  n/3$ and $\tau>n$ implies $\tau>k$, for $k\leq n$, we have
\begin{align*}
I &\leq \EE[1_{\{S>k\}}1_{\{\tau>k\}}(\frac{n}{3})^r\PP(T = k|\SC{F}_k)]\\
& \leq  \EE[1_{\{S>k\}}1_{\{\tau>k\}}(\frac{n}{3})^r\PP(\|M_{k+1}-M_k\| \geq  \frac{n}{3}|\SC{F}_k)].
\end{align*}
Notice that 
\begin{equation}
\label{probbound}\PP(\|M_{k+1}-M_k\| \geq  n/3 |\SC{F}_k) \leq 3^p n^{-p}E[\|M_{k+1}-M_k\|^p|\SC{F}_k] \leq 3^p n^{-p}\nu.
\end{equation}
It follows that $I \leq 3^{p-r} n^{r-p} \nu\PP(\tau>k).$

Next, observe that
\begin{align*}
II& \leq \EE[1_{\{\tau>k\}}\EE[\|M_{k+1} - M_k\|^r1_{\{\|M_{k+1}-M_k\|> n/3\}}|\SC{F}_k]].
\end{align*}
Then from \eqref{exp_formula} we have
\begin{align*}
\EE[\|M_{k+1} - M_k\|^r1_{\{\|M_{k+1}-M_k\|> n/3\}}|\SC{F}_k]& \leq \left(\f{ n}{3}\right)^r \PP(\|M_{k+1}-M_k\| \geq \f{ n}{3}|\SC{F}_k) \\
& \hspace{.4cm}+ \int_{\f{ n}{3}}^\infty ry^{r-1}\PP(\|M_{k+1}-M_k\| \geq y|\SC{F}_k) \ \drv y\\
& \leq \EE[\|M_{k+1}-M_k\|^p|\SC{F}_k] (3^{p-r} n^{r-p}\\
&\hspace{.4cm}+\int_{ n/3}^\infty ry^{r-1-p}\drv y)\\
&\leq \theta_2 n^{r-p}, \mbox{ for some } \theta_2.
\end{align*}
Hence, $II \leq \theta_2 n^{r-p} \PP(\tau>k)$.

Finally, 
\begin{align}
\label{III_cond}
III&\leq \EE[1_{\{\tau>k\}}1_{\{T=k\}}\EE[\|M_n-M_{k+1}\|^r|\SC{F}_{k+1}],
\end{align}
as by the definition $\{T=k\}$ is $\SC{F}_{k+1}$-measurable.
Notice that by  Lemma \ref{burkapp}  $\EE[\|M_n-M_{k+1}\|^p|\SC{F}_{k+1}] \leq  c_p\nu (n-k-1)^{p/2}$, for $k<n$.
Now since $r<p$, it follows that 
\begin{align*}
\EE[\|M_n-M_{k+1}\|^r|\SC{F}_{k+1}]& \leq \EE[\|M_n-M_{k+1}\|^p|\SC{F}_{k+1}]^{r/p}  \\
&\leq (c_p\nu)^{r/p} (n-k-1)^{r/2}, \mbox{ for } k<n\\
& \leq (c_p\nu)^{r/p} n^{r/2}, \mbox{ for } k<n.
\end{align*}
Putting this in \eqref{III_cond} we have for $k<n$
\begin{align*}
III&\leq (c_p\nu)^{r/p} n^{r/2}\EE[1_{\{\tau>k\}}1_{\{T=k\}}]\\
& = (c_p\nu)^{r/p}n^{r/2}\EE[1_{\{\tau>k\}}\PP(T=k|\SC{F}_k)]\\
&\leq (c_p\nu)^{r/p} n^{r/2}\EE[1_{\{\tau>k\}}\PP(\|M_{k+1} -M_k\| \geq  n/3|\SC{F}_k)]\\
& \leq (c_p\nu)^{r/p}3^p n^{r/2-p} \PP(\tau>k), \mbox{ by \eqref{probbound}}.
\end{align*}
Now for $k\geq 1$
\begin{align*}
\PP(\tau>k) \leq \PP(\|M_k\| \geq k)& \leq \EE[\|M_k\|^p]k^{-p} \\
&\leq \theta_0  k^{-p/2},\quad \mbox{by }\eqref{EMnbound}.
\end{align*}

Since $p>2$, it follows that for some $\theta_3$,
\begin{align*}
\EE[\|M_n\|^r1_{\{\tau>n\}}1_{\{T<S\}}]& \leq \theta_3 n^{r-p} (1+\sum_{k=1}^{n-1}\frac{\theta_0}{k^{p/2}})\\
&\leq  \theta_3 (1+\theta_0\zeta(p/2))n^{r-p},
\end{align*}
where $\zeta$ denotes the Riemann-zeta function.

\end{proof}

	\section*{Acknowledgments}
		It is a pleasure to thank Prof. Tom Kurtz and Federico Ramponi for helpful discussions.

	\bibliographystyle{elsarticle-num}
	\bibliography{ref}

\end{document}